\documentclass{amsart}
\usepackage{latexsym}
\usepackage{amsfonts,amssymb,amsmath}
\usepackage{amsthm}
\usepackage{graphicx,color}
\usepackage{multirow,bigdelim}
\usepackage{cleveref}

\newtheorem{theorem}{Theorem}	
\crefname{theorem}{Theorem}{Theorems}
\newtheorem{lemma}{Lemma}
\newtheorem{corollary}{Corollary}		
\newtheorem{proposition}{Proposition}	
\crefname{proposition}{Proposition}{Propositions}

\newtheorem{definition}{Definition}

\newtheorem{example}{Example}
\newtheorem{remark}{Remark}

\crefname{section}{Section}{Sections}
\crefname{theorem}{Theorem}{Theorems}
\crefname{lemma}{Lemma}{Lemmas}
\crefname{corollary}{Corollary}	{Corollaries}			
\crefname{proposition}{Proposition}{Propositions}	
\crefname{claim}{Claim}{Claims}
\crefname{conjecture}{Conjecture}{Conjectures}			
\crefname{definition}{Definition}{Definitions}
\crefname{problem}{Problem}{Problems}
\crefname{example}{Example}{Examples}
\crefname{remark}{Remark}{Remarks}
\crefname{figure}{Figure}{Figures}

\newcommand{\QED}{\hfill \ensuremath{\Box}}

\newcommand{\R}{\mathbb{R}}
\newcommand{\ld}{,\ldots ,}
\DeclareMathOperator{\rank}{rank}

\title[Characterization of generic transversality]
{
Characterization of generic transversality
}

\author{Shunsuke Ichiki
}
\address{
Institute of Mathematics for Industry, 
Kyushu University, 
Fukuoka 819-0395, 
Japan
}
\email{s-ichiki@imi.kyushu-u.ac.jp}

\thanks{The author was supported by JSPS KAKENHI Grant Numbers 
16J06911 and 19J00650.}

\begin{document}
\maketitle

\begin{abstract}
In this paper, the notion of generic transversality and its 
characterization are given. 
The characterization is also a further improvement of the 
basic transversality result and its strengthening 
which was given by John Mather. 
\end{abstract}


\section{Introduction}\label{sec:intro}
In this paper, unless otherwise stated, all manifolds are without boundary and 
assumed to have countable bases. 

Firstly, the definition of transversality is given. 
\begin{definition}\label{transverse}
{\rm
Let $X$ and $Y$ be $C^r$ manifolds, and $Z$ be a $C^r$ submanifold of $Y$ 
($r\geq1$). 
Let $f : X\to Y$ be a $C^1$ mapping. 
\begin{enumerate}
\item [(1)]
We say that $f:X\to Y$ is {\it transverse} to $Z$ 
{\it at $x$} if $f(x)\not\in Z$ or  
in the case of $f(x)\in Z$, the following holds: 
\begin{eqnarray*}
df_x(T_xX)+T_{f(x)}Z=T_{f(x)}Y.
\end{eqnarray*}
\item [(2)]
We say that $f:X\to Y$ is {\it transverse} to $Z$ 
if for any $x\in X$, 
the mapping $f$ is transverse to $Z$ at $x$. 
\end{enumerate}
}
\end{definition}
Let $X$, $A$ and $Y$ be $C^r$ manifolds ($r\geq1$). 
Let $U$ be an open set of $X\times A$.  
In the following, by $\pi_1 : U\to X$ and $\pi_2 : U\to A$, 
we denote the natural projections 
defined by 
\begin{eqnarray*}
\pi_1(x,a)&=&x, 
\\
\pi_2(x,a)&=&a. 
\end{eqnarray*} 
We say that a $C^1$ mapping $F:U\to Y$ is {\it generically transverse} to $Z$ if 
there exists a Lebesgue measure zero set $\Sigma$ of $\pi_2(U)$ such that 
for any $a\in \pi_2(U)-\Sigma$, the mapping $F_a:\pi_1(U\cap (X\times \{a\})) \to Y$ 
$(x\mapsto F(x,a))$ is transverse to $Z$. 
Here, note that $\pi_1(U\cap (X\times \{a\}))  $ is an open set of $X$. 
The main purpose of this paper 
is to 
give a characterization of generic transversality (for the main result, see \cref{main}). 

The following basic transversality result lies at the heart of most application of 
transversality.
\begin{proposition}[\cite{GG}]\label{abra}
Let $X$, $A$ and $Y$ be $C^\infty$ manifolds, $Z$ be a $C^\infty$ submanifold of $Y$ and $F:X\times A \to Y$ be a 
$C^\infty$ mapping. 
If $F$ is transverse to $Z$, 
then $F$ is generically transverse to $Z$. 
\end{proposition}

In \cite{GP}, 
an improvement of \cref{abra} is given by John Mather (for the result, see \cref{mather}).  
In order to state the result, we define the following. 
\begin{definition}\label{delta}
{\rm 
Let $X$ and $Y$ be $C^r$ manifolds, and $Z$ be a $C^r$ submanifold of $Y$ 
($r\geq1$). 
Let $f : X\to Y$ be a $C^1$ mapping. 
For any $x\in X$, set  
\begin{eqnarray*}
\delta(f, x ,Z)=\left\{ \begin{array}{ll}
0 & \text{if $f(x)\not \in Z$}, \\
\dim Y-\dim (df_x(T_xX)+T_{f(x)}Z) & \text{if $f(x)\in Z$}. \\
\end{array} \right.
\end{eqnarray*}
We define 
\begin{eqnarray*}
\delta(f,Z)={\rm sup} \{ {\delta(f, x ,Z) \mid x\in X} \}. 
\end{eqnarray*}
}
\end{definition}
In the case that all manifolds and mappings are of class $C^\infty$, 
\cref{delta} is the definition of \cite[p. 230]{GP}. 

As in \cite{brucekirk}, 
$\delta(f,x,Z)$ measures the extent to which $f$ fails to be transverse to 
$Z$ at $x$. 
It is clearly seen that $\delta(f,Z)=0$ if and only if 
$f$ is transverse to $Z$. 
The following result by Mather is a natural strengthening of \cref{abra}. 
\begin{theorem}[\cite{GP}] \label{mather}
Let $X$, $A$ and $Y$ be $C^\infty$ manifolds, 
$Z$ be a $C^\infty$ submanifold of $Y$ 
and $F:X\times A \to Y$ be a $C^\infty$ mapping. 
If for any $(x,a)\in X\times A$, $\delta(F_a, x, Z)=0$ or 
$\delta(F, (x,a), Z)<\delta(F_a, x, Z)$, 
then 
$F$ 
is generically transverse to $Z$. 
\end{theorem}

\cref{mather} is a useful tool for investigating global properties of mappings. 
For example, the result is an essential tool for the proofs of Theorem~1 in \cite{GP} and Theorem~2.2 in \cite{brucekirk}. 
However, it is difficult to apply to mappings 
with elements $(x,a)\in X\times A$ satisfying 
$\delta(F, (x,a), Z)=\delta(F_a, x, Z)>0$, 
while our main result \cref{main} dose work in this case.  
\begin{definition}\label{def:W}
{\rm 
Let $X$, $A$ and $Y$ be $C^r$ manifolds, and $Z$ be a $C^r$ submanifold of $Y$ 
($r\geq1$). 
Let $F:U\to Y$ be a $C^1$ mapping, 
where $U$ is an open set of $X\times A$. 
Then, we define  
\begin{eqnarray*}
W(F,Z)=\{(x,a)\in U \mid \delta(F_a, x, Z)=\delta(F, (x,a), Z)>0 \}.  
\end{eqnarray*}  
}
\end{definition}
\begin{theorem}\label{main}
Let $X$, $A$ and $Y$ be $C^r$ manifolds, 
$Z$ be a $C^r$ submanifold of $Y$ and $F:U\to Y$ be a $C^r$ mapping, where $U$ is an open set of $X\times A$. If 
\[
r>\max \{ \dim X+\dim Z-\dim Y +\delta(F,Z),0\},
\] 
then the following $(\alpha)$ and $(\beta)$ are equivalent. 
\begin{enumerate}
\item[$(\alpha)$]
The set $\pi_2(W(F,Z))$ has Lebesgue measure zero in $\pi_2(U)$, 
where $\pi_2:U\to A$ is the natural projection defined by $\pi_2(x,a)=a$.   
\item[$(\beta)$]
The mapping $F$ is generically transverse to $Z$. 
\end{enumerate}
\end{theorem}
Note this result works when the degree of differentiability is finite. 
%
From \cref{main}, in the case of $r=\infty$, we have the following. 
\begin{corollary}\label{main_coro}
Let $X$, $A$ and $Y$ be $C^\infty$ manifolds, 
$Z$ be a $C^\infty$ submanifold of $Y$ and $F:U\to Y$ be a $C^\infty$ mapping, where $U$ is an open set of $X\times A$. 
Then, the following $(\alpha)$ and $(\beta)$ are equivalent. 
\begin{enumerate}
\item[$(\alpha)$]
The set $\pi_2(W(F,Z))$ has Lebesgue measure zero in $\pi_2(U)$, 
where $\pi_2:U\to A$ is the natural projection defined by $\pi_2(x,a)=a$.   
\item[$(\beta)$]
The mapping $F$ is generically transverse to $Z$. 
\end{enumerate}
\end{corollary}


\begin{remark}\label{remark1}
{\rm  
\begin{enumerate}
\item [(1)]
If a given $C^\infty$ mapping $F:X\times A\to Y$ satisfies the assumption of  \cref{mather}, then it follows that $W(F,Z)=\emptyset$. 
Therefore, \cref{main} implies Mather's \cref{mather}. 
\item [(2)]
In \cref{main}, 
the hypothesis \[
r>\max \{ \dim X+\dim Z-\dim Y +\delta(F,Z),0\}
\] 
is used in the proof of $(\alpha)\Rightarrow(\beta)$. 
On the other hand, for the proof of 
$(\beta)\Rightarrow(\alpha)$, 
it is sufficient to assume that $r\geq1$. 

\item [(3)]
It is important to give the proof of \cref{main} for the following reason.  
The techniques for the proof of \cref{mather} in 
\cite[Lemma~2 (p. 230)]{GP} and \cite[Theorem~3.4~(p. 721)]{brucekirk} 
are significant  
for the proof of \cref{main}. 
However, in \cite[Lemma~2]{GP}, 
the proof of \cref{mather} is given only in the case that  
manifolds $X$ and $Z$ are compact, 
although the main results of \cite{GP} need the general case. 
The paper \cite{brucekirk} generalizes Mather's results on 
generic projections and also needs the transversality result. 
In \cite[Theorem~3.4]{brucekirk}, 
the general case is considered. However, there is an error in the assertion of Lemma~3.6 in \cite{brucekirk} (the set $\Sigma$ there need not be closed). 
In \cref{ex:3} of \cref{sec:W}, 
a counterexample of the assertion is given. 
Therefore, the current paper gives the first complete proof of \cref{mather}.  
\end{enumerate}
}
\end{remark}
In \cref{sec:W}, some examples of $W(F,Z)$ are given. 
In \cref{sec:lemma}, some assertions for the proof of \cref{main} 
are prepared. 
\cref{sec:main} is devoted to the proof of \cref{main}. 
\section{Some examples of $W(F,Z)$}\label{sec:W}
In this section, some examples of $W(F,Z)$ are given. 
\begin{example}
{\rm 
Let 
$F:\R\times \R\to \R$ be the mapping defined by $F(x,a)=0$, with $Z=\{0\}$. 
Then, for any 
$(x,a)\in \R\times \R$, it follows that $F(x,a)\in Z$ and  
$\delta(F, (x,a) ,Z)=\delta(F_a, x ,Z)=1$. 
Thus, we get 
\begin{eqnarray*}
W(F,Z)=\R\times \R. 
\end{eqnarray*}
From this example, we see that 
$W(F,Z)$ may not have Lebesgue measure zero. 
}
\end{example}

\begin{example}
{\rm 
Let 
$F:\R\times \R\to \R$ be the mapping defined by $F(x,a)=a^2x^2$, with $Z=\{0\}$. 
Then, $F(x,a)\in Z$ if and only if $ax=0$. 
We have 
\begin{eqnarray*}
JF_{(x,a)}&=&(2a^2x, 2ax^2), 
\\
(JF_a)_x&=&2a^2x. 
\end{eqnarray*}
Hence, we get  
\begin{eqnarray*}
\delta(F, (x,a) ,Z)
&=&
\left\{ 
\begin{array}{ll}
0 & \text{if $ax\not=0$}\\
1 & \text{if $ax=0$}, \\
\end{array} 
\right. 
\\
\delta(F_a, x ,Z)&=&
\left\{ 
\begin{array}{ll}
0 & ax\not=0\\
1 & ax=0. \\
\end{array} 
\right. 
\end{eqnarray*}
It follows that 
\begin{eqnarray*}
W(F,Z)=
\{(x,a)\in \R\times \R \mid ax=0 \}. 
\end{eqnarray*}
From this example, we see that 
$W(F,Z)$ may not be a manifold. 
}
\end{example}

\begin{example}\label{ex:3}
{\rm 
As in (3) of \cref{remark1}, 
the following is a counterexample to the assertion of Lemma 3.6 in \cite[p.~722]{brucekirk}. 

Let 
$F:\R\times \R\to \R^3$ be the mapping defined by $F(x,a)=(x,a,0)$. 
Let 
$Z=\{(x,0,0)\mid 0<x<1\}$ be a submanifold of $\R^3$. 
Then, 
$F(x,a)\in Z$ if and only if $(x,a)\in (0,1)\times \{0\}$. 
Here, $(0,1)$ is the open interval defined by $0<x<1$. 
We have  
\begin{eqnarray*}
\delta(F, (x,a) ,Z)
&=&
\left\{ 
\begin{array}{ll}
0 & \text{if $(x,a)\not \in (0,1)\times \{0\}$} \\
1 & \text{if $(x,a) \in (0,1)\times \{0\}$}, \\
\end{array} 
\right. 
\\
\delta(F_a, x ,Z)&=&
\left\{ 
\begin{array}{ll}
0 & \text{if $(x,a)\not \in (0,1)\times \{0\}$} \\
2 & \text{if $(x,a) \in (0,1)\times \{0\}$}. \\
\end{array} 
\right. 
\end{eqnarray*}
It follows that 
$W(F,Z)=\emptyset$. 
Set 
\begin{eqnarray*}
\Sigma=\{(x,a)\in \R\times \R  \mid \delta(F, (x,a), Z)=\delta(F,Z) \}. 
\end{eqnarray*}
Note that $F$ satisfies that $\delta(F_a, x ,Z)=0$ or 
$\delta(F, (x,a) ,Z)<\delta(F_a, x ,Z)$ for any $(x,a)\in \R\times \R$ 
and that $\delta(F,Z)$ is a positive integer. 
However, from $\Sigma=(0,1)\times \{0\}$, the set $\Sigma$ is not a closed set. 
}
\end{example}
\section{Assertions for the proof of \cref{main}}\label{sec:lemma}
In this section, some assertions for the proof of \cref{main} are prepared. 
\begin{lemma}\label{inequality}
Let $X$, $A$ and $Y$ be $C^r$ manifolds,  
$Z$ be a $C^r$ submanifold of $Y$ and 
$F:U\to Y$ be a $C^1$ mapping, 
where $U$ is an open set of $X\times A$ $(r\geq1)$. 
Then, it follows that 
\begin{eqnarray*}
\delta(F_a,x,Z)\geq \delta(F, (x,a), Z) 
\end{eqnarray*}
for any $(x,a)\in U$. 
\end{lemma}

\begin{proof}
Let $(x,a)\in U$ be any point. 
For simplicity, set $X'=\pi_1( U\cap(X\times \{a\}))$. 
It is not hard to see that 
\begin{eqnarray*}
(dF_{a})_{x}(T_{x}X') \subset 
dF_{(x,a)}(T_{(x,a)}U). 
\end{eqnarray*}
Hence,  we have 
\begin{eqnarray*}
&{}&\delta(F_{a},x,Z)- \delta(F, (x,a), Z) 
\\
&=&\left(\dim Y-\dim ((dF_{a})_{x}(T_{x}X')+T_{F_{a}(x)}Z)\right)
\\
&&-
\left(\dim Y-\dim (dF_{(x,a)}(T_{(x,a)}U)+T_{F(x,a)}Z)\right)
\\
&\geq &0. 
\end{eqnarray*}
\end{proof}
In the following,  
for two sets $V_1$, $V_2$, a mapping $f:V_1\to V_2$, and 
a subset $V_3$ of $V_1$, the restriction of the mapping $f$ to $V_3$ 
is denoted by $f|_{V_3}:V_3\to V_2$. 

Let $X$ and $Y$ be $C^r$ manifolds, and let $f : X\to Y$ be a $C^1$ mapping 
($r\geq1$). A point $x\in X$ is called a {\it critical point} of $f$ if
it is not a regular point, i.e., the rank of $df_x$ is less than the dimension of
$Y$. We say that a point $y\in Y$ is a {\it critical value} if it is the image of a critical point. 
A point $y\in Y$ is called a {\it  regular value} if it is not a critical value. 
The following is Sard's theorem. 
\begin{theorem}[\cite{sard}]\label{sard}
If $X$ and $Y$ are $C^r$ manifolds, 
$f : X \to Y$ is a $C^r$ mapping, and $r> \max\{\dim X-\dim Y,0\}$, then the set of
critical values of $f$ has Lebesgue measure zero.
\end{theorem}
The following result can be proved by the same argument 
as in the proof of \cref{abra}. 
For the sake of readers' convenience, the proof is given. 
\begin{proposition}[\cite{GG}]\label{lemma1r}
Let $X$, $A$ and $Y$ be $C^r$ manifolds, 
$Z$ be a $C^r$ submanifold of $Y$ and $F:U\to Y$ be a $C^r$ mapping, where $U$ is an open set of $X\times A$. If 
\[
r>\max \{ \dim X+\dim Z-\dim Y,0\},
\]
and $F$ is transverse to $Z$, 
then $F$ is generically transverse to $Z$. 
\end{proposition}
\begin{proof}
Since $F$ is transverse to $Z$, the set $F^{-1}(Z)$ is 
a $C^r$ submanifold of $U$ satisfying 
\begin{eqnarray}\label{eq:lemma1r}
\dim X+\dim A- \dim F^{-1}(Z)=\dim Y-\dim Z. 
\end{eqnarray}

Firstly, suppose that $\dim F^{-1}(Z)=0$. Then, 
since $F^{-1}(Z)$ is a countable set, 
$\pi_2(F^{-1}(Z))$ has Lebesgue measure zero in $\pi_2(U)$. 
It is clearly seen that for any $a\in\pi_2(U)-\pi_2(F^{-1}(Z))$, 
the mapping $F_a$ is transverse to $Z$. 

Finally, we will consider the case $\dim F^{-1}(Z)>0$. 
It is not hard to see that if $a\in \pi_2(U)$ is 
a regular value of $\pi_2|_{F^{-1}(Z)}$, then 
$F_a$ is transverse to $Z$. Here, $\pi_2:U\to A$ is the natural projection 
defined by $\pi_2(x,a)=a$ as in \cref{sec:intro}. 
Let $\Sigma$ be the set of critical values of $\pi_2|_{F^{-1}(Z)}$. 
From $r>\max \{ \dim X+\dim Z-\dim Y,0\}$ and 
(\ref{eq:lemma1r}), we have $r>\max \{ \dim F^{-1}(Z)-\dim A,0\}$. 
From \cref{sard}, $\Sigma$ has Lebesgue measure zero in $A$. 
Since $\pi_2(U)$ is an open set of $A$, the set $\Sigma\cap \pi_2(U)$ 
has Lebesgue measure zero in $\pi_2(U)$. 
Therefore, if $a\in \pi_2(U)-(\Sigma\cap \pi_2(U))$, then 
$F_a$ is transverse to $Z$. 
\end{proof}
The following lemma is a crucial result and the proof will be separated 
into three cases. 
\begin{lemma}\label{vital1}
Let $X$, $A$ and $Y$ be $C^r$ manifolds, 
$Z$ be a $C^r$ submanifold of $Y$ and 
$F:U\to Y$ be a $C^1$ mapping, 
where $U$ is an open set of $X\times A$ $(r\geq1)$. 
For any integer $\rho$ satisfying $0\leq \rho \leq \delta(F,Z)$, set 
\begin{eqnarray*}
\widetilde{W}_\rho=\{ (x,a)\in U \mid 
\delta(F_a,x,Z)>\delta(F, (x,a), Z)=\rho \}. 
\end{eqnarray*}
Then, for any $(x_0, a_0)\in \widetilde{W}_\rho$, 
there exist an open neighborhood $\widetilde{U}$ of $(x_0, a_0)$ 
and a $C^r$ submanifold $\widetilde{Z}$ of $Y$ 
satisfying the following:
\begin{enumerate}
\item [(1)]
$\dim \widetilde{Z}=\dim Z+\rho$. 
\item [(2)]
$F(\widetilde{U})\cap Z\subset \widetilde{Z}$. 
\item [(3)]
The mapping $F|_{\widetilde{U}}:\widetilde{U} \to Y$ is transverse to $\widetilde{Z}$. 
\item [(4)]
For any $(x,a)\in \widetilde{U}$, it follows that 
$\delta(F_a, x, Z)-\delta(F_a, x, \widetilde{Z})\leq \rho$. 
\end{enumerate}
\end{lemma}
\begin{proof}
In this proof, for a positive integer $k$, 
we denote the $k\times k$ unit matrix by $E_k$. 
Set $n=\dim X$, $m=\dim A$, $\ell=\dim Y$ and $q=\dim Z$. 

Let $(x_0,a_0)\in \widetilde{W}_\rho$ be any point. 
Then, we get $\rho<\ell$. 
Indeed, if $\rho\geq \ell$, then we have $\delta(F_{a_0},{x_0},Z)>\ell$. 
This contradicts $\delta(F_{a_0},{x_0},Z)\leq \ell$. 

From $\delta(F_{a_0},{x_0},Z)>0$, 
it is clearly seen that 
$F_{a_0}(x_0) (=F(x_0,a_0))\in Z$ and $q<\ell$.  

Let $(U',(x_1\ld x_n,a_1\ld a_m))$ (resp., $(V,(y_1\ld y_\ell))$) be a coordinate 
neighborhood containing $(x_0,a_0)\in U$ (resp., $F(x_0,a_0)\in Y$) 
such that 
\begin{eqnarray*}
Z\cap V&=&\{(y_1\ld y_q,y_{q+1}\ld y_\ell)\in V \mid y_{q+1}=\cdots =y_\ell=0\}, 
\\
F(U')&\subset &V, 
\\
F(x_0,a_0)&=&(0\ld0) \in \R^\ell.  
\end{eqnarray*} 

We consider the two cases:  
\begin{eqnarray*}
\left\{ \begin{array}{ll}
\text{1. The case $q=0$}. \\
\text{2. The case $q>0$}. \\
\end{array} \right.
\end{eqnarray*}

\underline{1. The case $q=0$.}
\vspace{1mm}

From $F(x_0,a_0)\in Z$ and $\dim T_{F(x_0,a_0)}Z=0$, we have 
\begin{eqnarray*}
\delta(F, (x_0,a_0), Z)
&=&\ell-\rank dF_{(x_0,a_0)}. 
\end{eqnarray*}
From $\delta(F, (x_0,a_0), Z)=\rho$, we get $\rank dF_{(x_0,a_0)}=\ell -\rho$. 
Set $F|_{U'}=(F_1\ld F_{\ell})$. 
For any $(x,a)\in U'$, set 
\begin{eqnarray*}
&{}&M_1(x,a) \\
&=&
\left(
\begin{array}{cccccc}
\frac{\partial F_{1}}{\partial x_1}(x,a) & \cdots & \frac{\partial F_{1}}{\partial x_n}(x,a) &\frac{\partial F_{1}}{\partial a_1}(x,a) & \cdots & \frac{\partial F_{1}}{\partial a_m}(x,a) 
\\ 
\vdots &  & \vdots & \vdots &  & \vdots 
\\ 
\frac{\partial F_{\ell-\rho}}{\partial x_1}(x,a) & \cdots & \frac{\partial F_{\ell-\rho}}{\partial x_n}(x,a) &\frac{\partial F_{\ell-\rho}}{\partial a_1}(x,a) & \cdots & \frac{\partial F_{\ell-\rho}}{\partial a_m}(x,a) 
\end{array}
\right).
\end{eqnarray*}
Here, note that $\ell -\rho>0$. 
Then, we have 
\begin{eqnarray*}
(JF)_{(x_0,a_0)}=
\left(
\begin{array}{c}
M_1(x_0,a_0) 
\\ \\[-7mm]
\\ 
\hline \\ \\[-7mm]
\ast
\end{array}
\right). 
\end{eqnarray*}
From $\rank dF_{(x_0,a_0)}=\ell -\rho$, 
without loss of generality, from the first 
we may assume that $\rank M_1(x_0,a_0)=\ell-\rho$. 
Since all entries of $M_1(x,a)$ are continuous functions 
of $U'$ into $\R$, 
there exists an open neighborhood $\widetilde{U}$ of $(x_0,a_0)$ 
such that $\rank M_1(x,a)\geq \ell-\rho$ for any $(x,a)\in \widetilde{U}$ 
and $\widetilde{U}\subset U'$. 
Set 
\begin{eqnarray*}
\widetilde{Z}=\{(y_1\ld y_\ell)\in V \mid y_{1}=\cdots =y_{\ell-\rho}=0\}. 
\end{eqnarray*}
Since $\widetilde{Z}$ is a $C^r$ submanifold of dimension $\rho$, 
we get the assertion (1). 

From $\widetilde{U}\subset U'$ and 
$F(U')\subset V$, we get 
$F(\widetilde{U})\subset V$. 
Hence,  
we have $F(\widetilde{U})\cap Z\subset V\cap Z$. 
From  $V\cap Z \subset \widetilde{Z}$, we have established (2). 

We now prove (3). 
Let $(x,a)\in \widetilde{U}$ be any point satisfying 
$F|_{\widetilde{U}}(x,a)\in \widetilde{Z}$. 
Then, we have 
\begin{eqnarray*}
\delta(F|_{\widetilde{U}},(x,a),\widetilde{Z})&=&
\dim Y-\dim \left(
(dF|_{\widetilde{U}})_{(x,a)}(T_{(x,a)}\widetilde{U})+
T_{F|_{\widetilde{U}}(x,a)}\widetilde{Z}
\right)
\\
&=&\ell -\rank M_2(x,a), 
\end{eqnarray*}
where 
\begin{eqnarray*}
M_2(x,a)=\left\{ \begin{array}{ll}
\left(
\begin{array}{c|c}
M_1(x,a) &O
\\ \\[-3mm]
\hline 
 \\[-2.5mm]
\ast & E_\rho
\end{array}
\right)  & \text{if $\rho>0$},  
\\
\\
M_1(x,a) & \text{if $\rho=0$}.  
\end{array} \right.
\end{eqnarray*}

Here, $O$ is the 
$(\ell-\rho)\times \rho$ 
zero matrix. 
From $\rank M_1(x,a)\geq \ell-\rho$, we get $\rank M_2(x,a)=\ell$. 
Namely, $\delta(F|_{\widetilde{U}},(x,a),\widetilde{Z})=0$. 
Hence, we get the assertion~(3). 

Finally, we will prove (4). 
Let $(x,a)\in \widetilde{U}$ be any point. 
Suppose that $F(x,a)\not \in Z$. In the case, 
from $\delta(F_a,x,Z)=0$, the assertion~(4) clearly holds. 
Now, suppose that $F(x,a) \in Z$. 
From the assertion~(2), we have $F(x,a) \in \widetilde{Z}$. 
Hence, it follows that 
\begin{eqnarray*}
&{}&\delta(F_{a},x_,Z)- \delta(F_a, x, \widetilde{Z}) 
\\
&=&\left(\dim Y-\dim ((dF_{a})_{x}(T_{x}X')+T_{F_{a}(x)}Z)\right)
\\
&&-
\left(\dim Y-\dim ((dF_{a})_x(T_x X')+T_{F_a(x)}\widetilde{Z})\right)
\\
&=&-\dim (dF_{a})_{x}(T_{x}X')
+\dim ((dF_{a})_x(T_x X')+T_{F_a(x)}\widetilde{Z})
\\
&\leq&-\dim (dF_{a})_{x}(T_{x}X')
+\dim (dF_{a})_x(T_x X')+\dim T_{F_a(x)}\widetilde{Z}
\\
&=&\rho,
\end{eqnarray*}
where $X'=\pi_1(U\cap (X\times \{a\}))$. 
Therefore, we have proved (4). 
\par 
\medskip 
\underline{2. The case $q> 0$.}\qquad 
\vspace{1mm}

From $q>0$ and $\ell-q>0$, set 
$F|_{U'}=(F_{11}\ld F_{1q}, F_{21}\ld F_{2,\ell-q})$. 
Set $F_1=(F_{11}\ld F_{1q})$ and $F_2=(F_{21}\ld F_{2,\ell-q})$. 

From $F(x_0,a_0)\in Z$, we have 
\begin{eqnarray*}
\delta(F,(x_0,a_0),Z)&=&\dim Y-\dim (dF_{(x_0,a_0)}(T_{(x_0,a_0)}U)+
T_{F_{(x_0,a_0)}}Z)\\
&=&\ell -\rank M_3(x_0,a_0), 
\end{eqnarray*}
where 
\begin{eqnarray*}
M_3(x,a)&=&
\left(
\begin{array}{c|c}
(JF_1)_{(x,a)} &E_q 
\\ \\[-3mm]
\hline 
 \\[-2.5mm]
(JF_2)_{(x,a)} & O
\end{array}
\right)
\end{eqnarray*}
and $(x,a)\in U'$. 
Here, $O$ is the $(\ell-q)\times q$ 
zero matrix. 
From $\delta(F,(x_0,a_0),Z)=\rho$ and 
$\rank M_3(x_0,a_0)=\rank (JF_2)_{(x_0,a_0)}+q$, 
we have 
\begin{eqnarray*}
\rank (JF_2)_{(x_0,a_0)}=\ell-q-\rho. 
\end{eqnarray*}
For the proof of the case $q>0$, it is sufficient to consider: 
\begin{eqnarray*}
\left\{ \begin{array}{ll}
\text{2.1. The case $q>0$ and 
$\rank (JF_2)_{(x_0,a_0)}=0$ $(\ell-q-\rho=0)$}. \\
\text{2.2. The case $q>0$ and 
$\rank (JF_2)_{(x_0,a_0)}>0$ $(\ell-q-\rho>0)$}. \\
\end{array} \right.
\end{eqnarray*}

\par 
\underline{2.1. The case $q>0$ and $\rank (JF_2)_{(x_0,a_0)}=0$ $(\ell-q-\rho=0)$. }\qquad 
\vspace{1mm}

Set $\widetilde{Z}=V$ and $\widetilde{U}=U'$. 
Then, the set $\widetilde{Z}$ is a $C^r$ open submanifold of $Y$. 
From $\dim \widetilde{Z}=\ell$ and $\ell-q-\rho=0$, 
we have $\dim \widetilde{Z}=\dim Z +\rho$. 
Thus, we have proved (1). 

From $F(\widetilde{U})$ $(=F(U') )\subset V$ and $\widetilde{Z}=V$, 
we get $F(\widetilde{U}) \cap Z\subset \widetilde{Z}$. 
Hence, we have the assertion~(2).  

Since $\widetilde{Z}$ is an open submanifold, the assertion~(3) holds. 

Finally, we prove (4). 
Let $(x,a)\in \widetilde{U}$ be any point. 
Since $\widetilde{Z}$ is an open submanifold, 
we get $\delta(F_a,x,\widetilde{Z})=0$. 
In the case of $F_a(x)\not\in Z$, we have 
$\delta(F_a,x,Z)=0$. Hence, in the case, 
the assertion~(4) holds. 
In the case of $F_a(x)\in Z$, 
we get 
\begin{eqnarray*}
\delta(F_a,x,Z)&=&\ell-\dim ((dF_a)_x(T_xX')+T_{F_a(x)}Z)
\\ 
&\leq& \ell - \dim T_{F_a(x)}Z
\\
&=&\ell-q
\\
&=&\rho, 
\end{eqnarray*}
where $X'=\pi_1(U\cap (X\times \{a\}))$. 
Therefore, (4) holds. 
\par 
\medskip 
\underline{2.2. The case $q>0$ and $\rank (JF_2)_{(x_0,a_0)}>0$ $(\ell-q-\rho>0)$. }
\vspace{1mm}

From $\ell-q-\rho>0$, for any $(x,a)\in U'$, 
set 
{\footnotesize
\begin{eqnarray*}
&{}&M_4(x,a) \\
&=&
\left(
\begin{array}{cccccc}
\frac{\partial F_{21}}{\partial x_1}(x,a) & \cdots & \frac{\partial F_{21}}{\partial x_n}(x,a)  &\frac{\partial F_{21}}{\partial a_1}(x,a)  & \cdots & \frac{\partial F_{21}}{\partial a_m}(x,a)  
\\ 
\vdots &  & \vdots & \vdots &  & \vdots 
\\ 
\frac{\partial F_{2,\ell-q-\rho}}{\partial x_1}(x,a)  & \cdots & \frac{\partial F_{2,\ell-q-\rho}}{\partial x_n}(x,a)  &\frac{\partial F_{2,\ell-q-\rho}}{\partial a_1}(x,a)  & \cdots & \frac{\partial F_{2,\ell-q-\rho}}{\partial a_m}(x,a)  
\end{array}
\right).
\end{eqnarray*}
}Then, we get 
\begin{eqnarray*}
(JF_2)_{(x_0,a_0)}=\left\{ \begin{array}{ll}
\left(
\begin{array}{c}
M_4(x_0,a_0) 
\\ \\[-7mm]
\\ 
\hline \\ \\[-7mm]
\ast
\end{array}
\right) & \text{if $\rho>0$},  
\\
\\
M_4(x_0,a_0)  & \text{if $\rho=0$}.  \\
\end{array} \right.
\end{eqnarray*}
From $\rank (dF_2)_{(x_0,a_0)}=\ell -q-\rho$, 
without loss of generality, from the first 
we may assume that $\rank M_4(x_0,a_0)=\ell-q-\rho$. 
Since all entries of $M_4(x,a)$ are continuous functions 
of $U'$ into $\R$, 
there exists an open neighborhood $\widetilde{U}$ of $(x_0,a_0)$ 
such that $\rank M_4(x,a)\geq \ell-q-\rho$ for any $(x,a)\in \widetilde{U}$ 
and $\widetilde{U}\subset U'$. 
Set 
\begin{eqnarray*}
\widetilde{Z}=\{(y_1\ld y_\ell)\in V \mid y_{q+1}=\cdots =y_{\ell-\rho}=0\}. 
\end{eqnarray*}
The set $\widetilde{Z}$ is a $C^r$ submanifold of $Y$. 
From $\dim \widetilde{Z}=q+\rho$, (1) holds. 

From $F(U')\subset V$, we have $F(\widetilde{U})\subset V$. 
Hence, it follows that 
$F(\widetilde{U})\cap Z\subset V\cap Z \subset \widetilde{Z}$. 
Thus, (2) holds. 

Next, we will prove (3). 
Let $(x,a)\in \widetilde{U}$ be any point satisfying 
$F|_{\widetilde{U}}(x,a)\in \widetilde{Z}$. 
Then, we have 
\begin{eqnarray*}
\delta(F|_{\widetilde{U}},(x,a),\widetilde{Z})&=&
\dim Y-\dim \left(
(dF|_{\widetilde{U}})_{(x,a)}(T_{(x,a)}\widetilde{U})+
T_{F|_{\widetilde{U}}(x,a)}\widetilde{Z}
\right)
\\
&=&\ell -\rank M_5(x,a), 
\end{eqnarray*}
where 
\begin{eqnarray*}
M_5(x,a)=\left\{ \begin{array}{ll}
\left(
\begin{array}{c|c|c}
(JF_1)_{(x,a)} &E_q & O
\\[1mm]
 \cline{1-3} 
 & & 
 \\[-3mm]
M_4(x,a) & O & O
\\[1mm]
 \cline{1-3} 
 & & 
 \\[-3mm]
\ast & O & E_{\rho}
\end{array}
\right)  & \text{if $\rho>0$},  
\\ 
\\
\left(
\begin{array}{c|c}
(JF_1)_{(x,a)} &E_q
\\[1mm]
 \cline{1-2} 
 & 
 \\[-3mm]
M_4(x,a) & O 
\end{array}
\right)& \text{if $\rho=0$}. \\
\end{array} \right.
\end{eqnarray*}
It follows that 
\begin{eqnarray*}
\rank M_5(x,a)&=&\rank M_4(x,a)+q+\rho
\\
&=&\ell. 
\end{eqnarray*}
From $\delta(F|_{\widetilde{U}},(x,a),\widetilde{Z})=0$, (3) holds. 

Finally, we prove (4). 
Let $(x,a)\in \widetilde{U}$ be any point. 
Suppose that $F(x,a)\not \in Z$. In the case, 
from $\delta(F_a,x,Z)=0$, the assertion~(4) clearly holds. 
Now, suppose that $F(x,a) \in Z$. 
From the assertion~(2), we have $F(x,a) \in \widetilde{Z}$. 
Hence, it follows that 
\begin{eqnarray*}
&{}&\delta(F_{a},x_,Z)- \delta(F_a, x, \widetilde{Z}) 
\\
&=&\left(\dim Y-\dim ((dF_{a})_{x}(T_{x}X')+T_{F_{a}(x)}Z)\right)
\\
&&-
\left(\dim Y-\dim ((dF_{a})_x(T_x X')+T_{F_a(x)}\widetilde{Z})\right)
\\
&=&-\dim ((dF_{a})_{x}(T_{x}X')+T_{F_{a}(x)}Z)
+\dim ((dF_{a})_x(T_x X')+T_{F_a(x)}\widetilde{Z})
\\
&\leq&\rho,
\end{eqnarray*}
where $X'=\pi_1(U\cap (X\times \{a\}))$. 
Therefore, we have proved (4). 
\end{proof}

\section{Proof of \cref{main}}\label{sec:main}
In this section, for simplicity, set 
\begin{eqnarray*}
W=W(F,Z). 
\end{eqnarray*}
\subsection{Proof of $(\alpha)\Rightarrow (\beta)$}\label{subsec:1}
Set 
\begin{eqnarray*}
\widetilde{W}=\{(x,a)\in U \mid \delta(F_a, x, Z)>\delta(F, (x,a), Z)\}.
\end{eqnarray*}
Then, we have 
\[
\pi_2(U)-\pi_2(W)\cup \pi_2(\widetilde{W})
=\{a\in \pi_2(U)\mid  F_a \mbox{ is transverse to } Z \}.
\leqno{(\ast)}    
\]

Firstly, we will show that for any $a\in \pi_2(U)-\pi_2(W)\cup \pi_2(\widetilde{W})$, 
the mapping $F_a$ is transverse to $Z$. 
Suppose that $F_a$ is not transverse to $Z$. 
Then, there exists an element $x\in \pi_1(U\cap (X\times \{a\})) $ satisfying 
$\delta(F_a,x,Z)>0$.  
From \cref{inequality}, 
it is not hard to see that 
$(x,a)\in W\cup\widetilde{W}$. 
Then, we get $a=\pi_2(x,a)\in \pi_2(W\cup\widetilde{W})$. 
This contradicts $a\in \pi_2(U)-\pi_2(W)\cup \pi_2(\widetilde{W})$. 

Next, we will show that 
for any $a\in \pi_2(U)$ such that $F_a$ is transverse to $Z$, 
we have $a\in \pi_2(U)-\pi_2(W)\cup \pi_2(\widetilde{W})$. 
Suppose that $a\in \pi_2(W)\cup \pi_2(\widetilde{W})$. 
Then, there exists an element $x\in \pi_1(U)$ satisfying 
$(x,a)\in W\cup \widetilde{W}$. 
Hence, the mapping $F_a$ is not transverse to $Z$. 
This contradicts the hypothesis that $F_a$ is transverse to $Z$. 
Thus, we get $(\ast)$. 

Now, set 
\begin{eqnarray*}
\Sigma=\pi_2(W)\cup \pi_2(\widetilde{W}). 
\end{eqnarray*}
From $(\ast)$, it is sufficient to show that 
$\Sigma$ has Lebesgue measure zero in $\pi_2(U)$. 
From the hypothesis, $\pi_2(W)$ has 
Lebesgue measure zero in $\pi_2(U)$. 
Hence, 
it is sufficient to show that 
$\pi_2(\widetilde{W})$ has Lebesgue measure zero in $\pi_2(U)$. 
Namely, for the proof of $(\alpha)\Rightarrow (\beta)$, 
it is sufficient to show the following. 
\begin{proposition}\label{main_sub}
Let $X$, $A$ and $Y$ be $C^r$ manifolds, 
$Z$ be a $C^r$ submanifold of $Y$ and $F:U\to Y$ be a $C^r$ mapping, where $U$ is an open set of $X\times A$. 
If \begin{eqnarray*} 
r>\max \{\dim X+\dim Z-\dim Y+\delta(F,Z),0\}, 
\end{eqnarray*}
then $\pi_2(\widetilde{W})$ has Lebesgue measure zero in $\pi_2(U)$, 
where  
\begin{eqnarray*}
\widetilde{W}=\{(x,a)\in U \mid \delta(F_a, x, Z)>\delta(F, (x,a), Z)\}.
\end{eqnarray*}
\end{proposition}
\noindent 
\\ 
\underline{Proof of \cref{main_sub}.}\quad 
\vspace{1mm}

Set 
\begin{eqnarray*}
\widetilde{W}_\rho=\{(x,a)\in U \mid \delta(F_a, x, Z)>\delta(F, (x,a), Z)=\rho \}. 
\end{eqnarray*}
We get 
\begin{eqnarray*}
\widetilde{W}=\bigcup_{0\leq \rho\leq \delta(F,Z)} \widetilde{W}_\rho. 
\end{eqnarray*}
In order to show that $\pi_2(\widetilde{W})$ 
has Lebesgue measure zero in $\pi_2(U)$, 
it is sufficient to show that $\pi_2(\widetilde{W}_\rho)$ has Lebesgue measure zero in 
$\pi_2(U)$ for any $\rho$ $(0\leq \rho\leq \delta(F,Z))$. 

From \cref{vital1}, 
there exist countably many open neighborhoods 
$\widetilde{U}_1,  \widetilde{U}_2, \ldots $ such that 
$\widetilde{W}_\rho \subset \cup_{i=1}^\infty  \widetilde{U}_i$ and 
countably many $C^r$ submanifolds 
$\widetilde{Z}_1, \widetilde{Z}_2,\ldots$ 
satisfying for any positive integer $i$,  
\begin{enumerate}
\item [(1)]
$\dim \widetilde{Z}_i=\dim Z +\rho$. 
\item [(2)]
$F(\widetilde{U}_i)\cap Z \subset \widetilde{Z}_i$. 
\item  [(3)]
The mapping 
$F|_{\widetilde{U}_i} :\widetilde{U}_i \to Y$ is transverse to $\widetilde{Z}_i$. 
\item [(4)]
For any $(x,a)\in \widetilde{U}_i$, it follows that  
$\delta(F_a, x, Z)- \delta(F_a,x,\widetilde{Z}_i)\leq \rho$. 
\end{enumerate}
From $\widetilde{W}_\rho \subset \cup_{i=1}^\infty  \widetilde{U}_i$, 
in order to show that $\pi_2(\widetilde{W}_\rho)$ is 
Lebesgue measure zero in $\pi_2(U)$, it is sufficient to show that 
for any $i$, the set 
$\pi_2(\widetilde{W}_\rho \cap \widetilde{U}_i)$ has Lebesgue measure zero in 
$\pi_2(\widetilde{U}_i)$. 

From $\rho \leq \delta(F,Z)$ and the assertion (1), 
we get 
\begin{eqnarray*}
r&>&\max \{\dim X+\dim Z+\delta(F,Z)-\dim Y,0\}
\\
&\geq& \max \{\dim X+\dim Z+\rho-\dim Y,0\}
\\
&=& \max \{\dim X+\dim \widetilde{Z}_i-\dim Y,0\}. 
\end{eqnarray*}
From the assertion~(3), 
we can apply \cref{lemma1r} to $F|_{\widetilde{U}_i}$. 
Hence, there exists 
 a Lebesgue measure zero set $\Sigma_i $ in 
$\pi_2(\widetilde{U}_i)$ 
such that 
for any $a\in \pi_2(\widetilde{U}_i)-\Sigma_i$, the mapping 
$(F|_{\widetilde{U}_i})_a$ is transverse to $\widetilde{Z}_i$. 
In order to finish the proof, it is sufficient to show that 
$\pi_2(\widetilde{W}_\rho \cap \widetilde{U}_i)\subset \Sigma_i$. 

Let $a\in \pi_2(\widetilde{W}_\rho \cap \widetilde{U}_i)$ be 
any element. Then, there exists an element  
$x\in \pi_1(\widetilde{U}_i)$ such that 
$(x,a)\in \widetilde{W}_\rho \cap \widetilde{U}_i$.   
By $(x,a)\in \widetilde{W}_\rho$, we get $\delta(F_a, x, Z)>\rho$. 
From the assertion~(4), 
\begin{eqnarray*}
\delta(F_a, x, Z)>\rho\geq \delta(F_a, x, Z)- \delta(F_a,x,\widetilde{Z}_i). 
\end{eqnarray*}
Hence, it follows that  
\begin{eqnarray*}
\delta(F_a,x,\widetilde{Z}_i)>0. 
\end{eqnarray*}
Namely, $(F|_{\widetilde{U}_i})_a$ is not transverse to $\widetilde{Z}_i$. 
Hence, we have $a \in \Sigma_i$. \QED
\subsection{Proof of $(\beta)\Rightarrow(\alpha)$}\label{subsec:2}
From $(\beta)$, there exists a Lebesgue measure zero set $\Sigma$ of $\pi_2(U)$ 
such that for any $a\in\pi_2(U)-\Sigma$, 
the mapping $F_a$ is transverse to $Z$. 

Suppose that $\pi_2(W)$ does not have Lebesgue measure zero in $\pi_2(U)$.
Then, it is clearly seen that $\pi_2(W) \not\subset \Sigma$. 
Thus, there exists an element $a \in \pi_2(W)$ satisfying 
$a \in \pi_2(U)-\Sigma$. 
From $a \in \pi_2(W)$, there exists an element $x\in \pi_1(U)$ such that $(x,a)\in W$. 
Hence, we get $\delta(F_a,x,Z)>0$. 
Namely, $F_a$ is not transverse to $Z$. 
This contradicts $(\beta)$. 
\QED
\section*{Acknowledgements}
The author is most grateful to the anonymous reviewer for carefully reading 
the first manuscript of this paper and for giving invaluable suggestions. 
The author is grateful to Kenta Hayano, Takashi Nishimura and Osamu Saeki 
for their kind comments. 

\end{document}